\documentclass{amsart}
\usepackage{setspace, amsmath, amsthm, amssymb, amsfonts, amscd, epic, graphicx, ulem, dsfont}
\usepackage[T1]{fontenc}
\usepackage{multirow}
\usepackage{bbm}
\usepackage{enumerate}

\makeatletter \@namedef{subjclassname@2010}{
  \textup{2010} Mathematics Subject Classification}
\makeatother

\newtheorem{thm}{Theorem}[section]
\newtheorem{cor}[thm]{Corollary}
\newtheorem{lem}[thm]{Lemma}
\newtheorem{pro}[thm]{Proposition}

\theoremstyle{remark}
\newtheorem*{rema}{Remark}

\theoremstyle{definition}
\newtheorem*{defn}{Definition}

\newcommand{\ran}{\text{\rm{ran}}}

\newcommand{\R}{\mathbb{R}}

\newcommand{\C}{\mathbb{C}}

\begin{document}

\title[Invertibility of Normal Operators and Applications]{Right (Or Left) Invertibility of Bounded and Unbounded Operators and Applications to the Spectrum of Products}
\author[Souheyb Dehimi and M. H. MORTAD]{Souheyb Dehimi and MOHAMMED HICHEM MORTAD$^*$}

\dedicatory{}
\thanks{* Corresponding author.}
\date{}


 \address{ Department of
Mathematics, University of Oran 1, A. Ben Bella, B.P. 1524, El
Menouar, Oran 31000, Algeria.\newline {\bf Mailing address}:
\newline Pr Mohammed Hichem Mortad \newline BP 7085 Seddikia Oran
\newline 31013 \newline Algeria}

\email{sohayb20091@gmail.com} \email{mhmortad@gmail.com,
mortad@univ-oran.dz.}

\begin{abstract}

This paper is mainly concerned with proving $\sigma(AB)=\sigma(BA)$
for two linear and non necessarily bounded operators $A$ and $B$.
The main tool is left and right invertibility of bounded and
unbounded operators.
\end{abstract}

\maketitle

\section{Introduction}
All operators considered here are linear and defined on a complex
separable Hilbert space $H$. In order to avoid trivialities in the
bounded case, we further assume that $\dim H=\infty$. Also, we
assume that the reader is well aware of the basic notions of bounded
and unbounded operators as well as the algebraic notions of right
and left invertibility.

It is known that if no condition is imposed on either of the
operators $A$ or $B$, then we are only sure that:

\[\sigma(AB)-\{0\}=\sigma(BA)-\{0\}.\]

We would like to know when

\[\sigma(AB)=\sigma(BA)\cdots\cdots(E)\]

holds for two linear bounded operators.

We already know that if one of the operators is invertible, then it
may be shown that $AB$ and $BA$ are similar, hence they have the
same spectrum. $(E)$ is also satisfied when one of the operators is
compact.

Hladnik-Omladi\v{c}
\cite{Hladnik-Omladic-spectrum-product-PAMS-1988} proved the
following:

\begin{thm}\label{hladnik} Let $A,B\in B(H)$ be such that $B$
is positive. Let $P$ be the (unique) square root of $B$. Then
\[\sigma(AB)=\sigma(BA)=\sigma(PAP).\]
\end{thm}

Another case for which the equality $\sigma(AB)=\sigma(BA)$ holds is
when one of the operators is normal:

\begin{thm}\label{baraa}(Barraa-Boumazghour, \cite{Barraa-Boumazghour})\label{spectrum AB BA normal B equality}
Let $A,B\in B(H)$ be such that one of them is normal. Then
\[\sigma(AB)=\sigma(BA).\]
\end{thm}

The proof of the preceding theorem relied on the following:

\begin{thm}\label{Spain}(Spain, \cite{Spain-normal-invertible}) A normal
unilaterally invertible element of a complex unital Banach algebra
is invertible.
\end{thm}

But, as quoted in \cite{Spain-normal-invertible} the result "appears
to have escaped notice up to now". However, Spain
\cite{Spain-normal-invertible} did miss that in Conway's \cite{Con}
(Fredholm Theory Chapter).

The first observation in our work is that all of the three previous
results become just a mere consequence (at least in $B(H)$!) of the
next proposition whose simple proof is left to the reader.

\begin{pro}\label{self-adjoint left right inve is invert exo}
Let $A,B\in B(H)$ be such that $A$ is self-adjoint. If $AB=I$ (or
$BA=I$), then $A$ is invertible and $B$ is self-adjoint.
\end{pro}

\begin{cor}\label{jjj}
 Let $A,B\in B(H)$ be such that $B$
is positive. Let $P$ be the (unique) square root of $B$. Then
\[\sigma(AB)=\sigma(BA)=\sigma(PAP).\]
\end{cor}

\begin{proof}
To establish $\sigma(AB)=\sigma(BA)$, we have to show that $AB$ is
invertible iff $BA$ is invertible. This is done as follows:

\begin{align*}
BA \text{ \textbf{is invertible} }&\Longrightarrow B \text{ is right invertible}\\
&\Longrightarrow B \text{ is invertible (by Proposition \ref{self-adjoint left right inve is invert exo})}\\
&\Longrightarrow B^{-1}\text{ is invertible }\\
&\Longrightarrow B^{-1}BA=A\text{ is invertible }\\
&\Longrightarrow AB \text{ \textbf{is invertible} }\\
&\Longrightarrow B \text{ is left invertible }\\
&\Longrightarrow B \text{ is invertible (by Proposition \ref{self-adjoint left right inve is invert exo})}\\
&\Longrightarrow B^{-1}\text{ is invertible }\\
&\Longrightarrow A=ABB^{-1} \text{ is invertible }\\
&\Longrightarrow BA \text{ \textbf{is invertible}. }
\end{align*}
This settles the first equality. To prove the second equality, just
apply the first equality to obtain
\[\sigma(BA)=\sigma(PPA)=\sigma(PAP).\]
\end{proof}

Using the polar decomposition of a normal operator, Proposition
\ref{self-adjoint left right inve is invert exo} yields

\begin{cor}\label{normal left right inve is invert exo} Let $A\in B(H)$ be a
right (or left) invertible normal operator. Then $A$ is invertible.
\end{cor}

\begin{proof}
Left to the reader.
\end{proof}

This (combined with Corollary \ref{jjj}) gives

\begin{cor}\label{spectrum AB BA normal B equality}
Let $A,B\in B(H)$ be such that one of them is normal. Then
\[\sigma(AB)=\sigma(BA).\]
\end{cor}

Unfortunately, we cannot go up to the class of hyponormal operators.
Indeed, consider the usual (unilateral) shift $S$ on $\ell^2$. Then
$S^*S=I$, $SS^*\neq I$ and $S$ is hyponormal. Hence
\begin{enumerate}
  \item $S$ is left invertible without being invertible;
  \item Also,
  \[\sigma(S^*S)=\{1\}\neq \sigma(SS^*)=\{0,1\}.\]
\end{enumerate}

We can, however, generalize the previous results to non necessarily
bounded operators. Moreover, normality is not indispensable. Only a
condition of the type $\ker(A)=\ker(A^*)$ (or even an inclusion in
some cases) will suffice. See Theorem \ref{normal unbd left right
inve is invert thm}.

It is worth noticing, that the works on the spectra of unbounded
products are only numbered. For instance, see
\cite{Hardt-Konstantinov-Spectrum-product}, \cite{Hardt-Mennicken},
\cite{Moller-essential spectrum} and \cite{Friedrich-Ran-Wojtylak}

We conclude this introduction with an application of Corollary
\ref{jjj} (cf. \cite{Mortad-PAMS2003}).

\begin{pro}(cf. Theorem \ref{Theorem hypo})
Let $A,B\in B(H)$ be such that $A$ is positive and $B$ is
self-adjoint. If $AB$ (or $BA$) is hyponormal, then $AB$ (or $BA$)
is self-adjoint.
\end{pro}

\begin{proof}
By Corollary \ref{jjj}, $\sigma(AB)$ (or $\sigma(BA)$) is real. The
result then follows by remembering that a hyponormal operator with a
real spectrum is self-adjoint (see e.g. \cite{Stampfli hyponormal
1965}).
\end{proof}

\begin{cor}
Let $A,B\in B(H)$ be such that $A$ is positive and $B$ is
self-adjoint. If $AB$ (or $BA$) is hyponormal, then $A+iB$ is
normal.
\end{cor}

\section{Left or right invertible unbounded operators}

First, recall (cf. \cite{Con}):

\begin{defn}
An unbounded linear operator $A$ with domain $D(A)\subset H$, is
said to be invertible  if there exists an everywhere defined $B\in
B(H)$ such that
\[AB=I \text{ and } BA\subset I.\]
\end{defn}

\begin{rema}
It is known that if $A$ and $B$ are two unbounded and invertible
operators, then $AB$ is invertible and $(AB)^{-1}=B^{-1}A^{-1}$.
\end{rema}

\begin{rema}
The invertibility of $A$ is equivalent to requiring $A$ to be
injective and $A^{-1}\in B(H)$. Hence, if $A$ is invertible, then
$A$ is closed and if $A$ is closed and densely defined, then $A$ is
invertible iff $A^*$ is so.
\end{rema}

Based on the definition just above and the bounded case, we
introduce:

\begin{defn}Let $H$ be a Hilbert space and let $A$ be an unbounded operator with domain $D(A)\subset
H$. We say that $A$ is \textbf{right invertible} if there exists an
everywhere defined $B\in B(H)$ such that $AB=I$; and we say that $A$
is \textbf{left invertible} if there is an everywhere defined $C\in
B(H)$ such that $CA\subset I$.
\end{defn}

\begin{rema}
It is easily seen that if $A$ is closed and densely defined, then
$A$ is left (resp. right) invertible iff $A^*$ is right (resp. left)
invertible.
\end{rema}

To show the importance of the notion of left or right invertibility,
we give the following result:

\begin{pro}\label{invert unboun left=right pro}
Let $A$ be a non necessarily bounded operator with domain
$D(A)\subset H$. If $A$ is left and right invertible simultaneously,
then $A$ is invertible.
\end{pro}

\begin{proof}By assumption $AB=I$ and $CA\subset I$ for some bounded $B,C\in
B(H)$. Then
\[C=C(AB)=(CA)B\subset B,\]
and hence $B=C$ as they are both everywhere defined.
\end{proof}

\begin{cor}\label{right or left invertible s.a. opera pro}
A right (or left) invertible non necessarily bounded self-adjoint
operator is invertible.
\end{cor}

We now turn to non necessarily bounded normal operators.
Fortunately, Corollary \ref{normal left right inve is invert exo}
also holds for unbounded operators. In fact, the result is true for
a more general class of operators.

\begin{thm}\label{normal unbd left right inve is invert thm}
A right (resp. left) invertible closed and densely defined operator
$A$ such that $\ker(A)\subseteq \ker (A^*)$ (resp.
$\ker(A^*)\subseteq \ker(A)$) is invertible. In particular, if $A$
is closed and densely defined and $\ker(A)=\ker (A^*)$, then $A$ is
left invertible iff $A$ is right invertible iff $A$ is invertible.
\end{thm}

\begin{proof}By the two remarks just above, it suffices to consider
the case of right invertibility. So assume that $A$ is right
invertible, i.e. $AB=I$ for some $B\in B(H)$. Hence
\[\ran(A)=H\Longrightarrow \ker(A^*)=\{0\}\Longrightarrow \ker(A)=\{0\}.\]
Since $A^{-1}$ is closed and $D(A^{-1})=\ran(A)=H$, the Closed Graph
Theorem yields $A^{-1}\in B(H)$, that is, $A$ is invertible.
\end{proof}

\begin{rema}
Plainly, self-adjoint and normal operators $A$ are closed densely
defined and they obey $\ker(A)=\ker(A^*)$.

On the other hand, if $A$ is closed and hyponormal, then
$\ker(A)\subseteq \ker(A^*)$. So a right invertible closed
hyponormal operator is invertible. Similarly, a left invertible
closed cohyponormal operator is invertible.
\end{rema}

\section{Spectra of Products of Unbounded Operators}

In this paper, we use the following definition (cf. \cite{Kat}) of
the spectrum:

\begin{defn}
Let $A$ be a non necessarily bounded operator with domain
$D(A)\subset H$. We say that $\lambda$ is not in $\sigma(A)$ if
$A-\lambda$ is injective and $(A-\lambda)^{-1}$ is in $B(H)$.
\end{defn}

\begin{rema}
Using the previous definition, we easily see that if
$\sigma(A)\neq\C$, then $A$ is closed.
\end{rema}

In \cite{Hardt-Konstantinov-Spectrum-product}, it is shown that if
$A$ and $B$ are two non necessarily bounded operators such that
$\sigma(AB)\neq\C$ and $\sigma(BA)\neq \C$ (hence both $AB$ and $BA$
are closed), then
\[\sigma(AB)-\{0\}=\sigma(BA)-\{0\}.\]

It is clear that if we want to obtain the equality
$\sigma(AB)=\sigma(BA)$, we must show that $AB$ is invertible iff
$BA$ is invertible. We reserve a substantial part to this
equivalence.

\begin{thm}\label{BA invertible implies AB invertible thm}
Assume  $A$ is a closed and densely defined operator in $H$ and
$B\in B(H)$ is such that $BA$ is invertible. If either $\ker
(A^*)\subseteq \ker(A)$ or $\ker(B)\subseteq\ker(B^*)$, then the
operators $A$, $B$ and consequently $AB$ are invertible.
\end{thm}

To prove this theorem we need a lemma.

\begin{lem}\label{BA invertible B right invertible lem}
If $A$ is closed and $B$ is an operator such that $BA$ is right
invertible, then $B$ too is right invertible.
\end{lem}

\begin{proof}
Since $BAC=I$ for some $C\in B(H)$, it follows that $D(AC)=H$. Hence
by the Closed Graph Theorem, $AC\in B(H)$, and we are done.
\end{proof}

Now we give a proof of Theorem \ref{BA invertible implies AB
invertible thm}.

\begin{proof}\hfill
\begin{itemize}
  \item $\ker(A^*)\subseteq \ker(A)$: We may write
  \begin{align*}
  BA \text{ invertible}\Longrightarrow & DBA\subseteq I \text{ for some }D\in B(H)\\
  \Longrightarrow & A \text{ left invertible}\\
  \Longrightarrow & A^{-1}\in B(H)\text{ (Theorem \ref{normal unbd left right inve is invert thm})}\\
  \Longrightarrow &B=(BA)A^{-1} \text{ injective}\\
  \Longrightarrow &A^{-1}B^{-1}=(BA)^{-1}\in B(H)\\
  \Longrightarrow &D(B^{-1})=H.
  \end{align*}
  In fine, the Closed Graph Theorem gives $B^{-1}\in B(H)$, as
  required.
  \item $\ker(B)\subseteq\ker(B^*)$: We can write:
  \begin{align*}
  BA \text{ invertible}\Longrightarrow & B \text{ right invertible (Lemma \ref{BA invertible B right invertible lem})}\\
  \Longrightarrow & B \text{ invertible (Theorem \ref{normal unbd left right inve is invert thm})}\\
  \Longrightarrow &B^{-1} \text{ invertible }\\
  \Longrightarrow &A=B^{-1}(BA) \text{ invertible. }\\
  \end{align*}
  Accordingly, $A$, $B$ and $AB$ are all invertible.
\end{itemize}
\end{proof}

Interchanging the roles of $BA$ and $AB$ in the assumptions of the
foregoing theorem does not lead to the invertibility of $BA$. An
extra condition has to be added. We have:

\begin{thm}\label{Ab invert implies BA invertible thm}
Assume $A$ is a closed densely defined operator and $B\in B(H)$ is
such that $B^*A^*$ is closed and $AB$ is invertible. If either $\ker
(A)\subseteq \ker(A^*)$ or $\ker(B^*)\subseteq\ker(B)$, then the
operators $A$, $B$ and consequently $BA$ are invertible.
\end{thm}

\begin{proof}This is a consequence of Theorem \ref{BA invertible implies AB invertible
thm}. Indeed, since $AB=(B^*A^*)^*$ is invertible and $B^*A^*$ is
closed and densely defined, we infer from the second remark in the
beginning of Section 2 that $B^*A^*$ is invertible. Applying Theorem
\ref{BA invertible implies AB invertible thm}, we see that $A^*$ and
$B^*$ are invertible, and by the same remark again so are $A$ and
$B$.

\end{proof}

\begin{cor}\label{sigma(AB)=sigma(BA) cor}
Let $A$ be a closed densely defined operator in $H$ and let $B\in
 B(H)$ be such that $\sigma(B^*A^*)\neq\C$ and $\sigma(BA)\neq \C$. If
 either $\ker
(A^*)=\ker(A)$ or $\ker(B)=\ker(B^*)$, then
\[\sigma(BA)=\sigma(AB).\]
\end{cor}

\begin{proof}Since $\sigma(AB)=\sigma[(B^*A^*)^*]=[\sigma(B^*A^*))]^*\neq
\C$, we see that the required result follows from Theorem \ref{BA
invertible implies AB invertible thm}, Theorem \ref{Ab invert
implies BA invertible thm} and
\cite{Hardt-Konstantinov-Spectrum-product}.
\end{proof}

\begin{rema}
A condition of the type $\sigma(BA)\neq \C$ is not unnatural.
Remember that if $A,B\in B(H)$, then we always have $\sigma(BA)\neq
\C$!
\end{rema}

\begin{cor}
Let $A$ and $B$ be two self-adjoint operators such that $B$ is
bounded. If $\sigma(BA)\neq \C$, then
\[\sigma(AB)=\sigma(BA).\]
\end{cor}

We finish this paper with the following result (cf.
\cite{Mortad-PAMS2003}).

\begin{thm}\label{Theorem hypo}
Let $A$ and $B$ be two self-adjoint operators such that $B$ is
bounded and positive. If $BA$ is hyponormal, then both $BA$ and $AB$
are self-adjoint (and $AB=BA$!) whenever $\sigma(BA)\neq \C$.
\end{thm}

\begin{rema}
The foregoing theorem was first shown with the extra assumption "$B$
being injective". Then, we discussed with Professor Jan Stochel
whether the closedness of $P^2A$ would imply that of $PA$, whenever
$P\in B(H)$ is self-adjoint? The answer turned out to be positive
and here is the result:
\end{rema}

\begin{pro}\label{P2A closed PA closed without inj! pro}
Let $P\in B(H)$ be self-adjoint and let $A$ be an arbitrary operator
such that $P^2A$ is closed. Then $PA$ is closed.
\end{pro}

\begin{proof}Let $(x_n)$ be in $D(PA)=D(A)$ such that
  \[PAx_n\longrightarrow y \text{ and } x_n\longrightarrow x.\]
Then it is clear that $y\in \overline{\ran(P)}$. Since $P$ is
continuous, we obtain
\[P^2Ax_n\longrightarrow Py \text{ and } x_n\longrightarrow x.\]
As $P^2A$ is closed, we then obtain
  \[P^2Ax=Py \text{ and } x\in D(P^2A)=D(A).\]
Hence $P(y-PAx)=0$, that is, $y-PAx\in \ker (P)$. Since also
$y-PAx\in \overline{\ran (P)}$ and $P$ is self-adjoint, we get
$y-PAx\in [\ker (P)]^{\perp}$. Thus $y-PAx=0$ or $PAx=y$. Since we
already know that $x\in D(A)=D(PA)$, the proof of the closedness of
$PA$ is complete.
\end{proof}

Now, we prove Theorem \ref{Theorem hypo}:

\begin{proof} Let $P$ be the unique square root of
$B$. Since $\sigma(BA)\neq \C$, $BA$ or $P^2A$ is closed so that
$PA$ is closed by Proposition \ref{P2A closed PA closed without inj!
pro}. The rest of the proof is divided into two parts.
\begin{enumerate}
  \item First, $PAP$ is self-adjoint: Since $P$ is bounded and $PA$ is closed, we have
  \[(PAP)^*=(AP)^*P^*=(AP)^*P=(PA)^{**}P=\overline{PA}P=PAP,\]
  i.e. $PAP$ is surely self-adjoint so that $\sigma(PAP)\neq\C$.
  \item Second, we show that $BA$ and $AB$ are self-adjoint: Since $\sigma(P^2A)\neq\C$ and $\sigma(PAP)\neq \C$,
 by Corollary \ref{sigma(AB)=sigma(BA) cor}
  \[\sigma(BA)=\sigma(PAP)\subset \R.\]
  Now, if $W(BA)$ denotes the numerical range of $BA$, then from
  \cite{Janas-Hyponormal-unbd-III} we know that
  \[W(BA)\subset \text{conv }\sigma(BA)\subset\R\]
  for $BA$ is hyponormal. Thus $BA$ is closed, symmetric and with real
  spectrum, it is self-adjoint! Accordingly,
  \[AB=(BA)^*=BA.\]
\end{enumerate}
\end{proof}

\begin{rema}
If we assume that $BA$ is subnormal (which is \textit{stronger} than
hyponormal), then we can obtain the self-adjointness of $BA$ and
$AB$ without using the machinery of the preceding proof, we just
apply Theorem 4.2 of \cite{STO}, and other known properties.
\end{rema}

\section{Conclusion}
It was the referee's idea to improve the results in the case of
unbounded operators by using conditions on kernels. Indeed, in the
first version of the paper we only dealt with normal and
self-adjoint operators. Needless to say that some of the results in
the bounded case are particular cases of some of those of their
unbounded counterparts.

\section*{Acknowledgement}

The authors wish to warmly thank the referee for his suggestions and
for his meticulousness.

\end{document}